\documentclass[11pt,a4paper]{amsart}
\usepackage{amssymb,amsmath,epsfig,graphics,mathrsfs}

\usepackage{fancyhdr}
\usepackage{hyperref}
\hypersetup{
 colorlinks   = true,
 urlcolor     = blue,
 linkcolor    = blue,
 citecolor   = red ,
 bookmarksopen=true
}

\usepackage{amsmath}
\usepackage{amsfonts}
\usepackage{amssymb}
\usepackage{amsthm}
\usepackage{epsfig,graphics,mathrsfs}
\usepackage{graphicx}

\usepackage[usenames, dvipsnames]{color} 

\usepackage{hyperref}

\def \de {\partial}

\def \phi {\varphi}
\def \RNu {\mathbb{R}^{N+1}}
\def \RN {\mathbb{R}^N}
\def \R {\mathbb{R}}

\def \K {\mathscr{K}}

\def \G{\Gamma}

\def \vf{\varphi}
\def \S {\mathscr{S}(\R^{N+1})}
\def \So {\mathscr{S}}

\newcommand{\sA}{\mathscr A}

\newcommand{\Bpa}{\mathfrak B^{\alpha,p}\left(\RNu\right)}

\newcommand{\Rn}{\mathbb R^n}

\newcommand{\Om}{\Omega}

\newcommand{\p}{\partial}

\numberwithin{equation}{section}

\newcommand{\beq}{\begin{equation}}
\newcommand{\bea}[1]{\begin{array}{#1} }
\newcommand{\eeq}{ \end{equation}}
\newcommand{\ea}{ \end{array}}

\newcommand{\Ks}{\left(-\K\right)^s}

\def \tr{\mathrm{tr}}

\newtheorem{theorem}{Theorem}[section]

\newtheorem{proposition}[theorem]{Proposition}

\newtheorem{definition}[theorem]{Definition}

\numberwithin{equation}{section}

\begin{document}

\title[A chain rule, etc.]{A chain rule for a class of evolutive nonlocal hypoelliptic equations}

\subjclass[2010]{}

\date{}

\begin{abstract}
We prove a chain rule of local type for a class of fractional hypoelliptic equations of Kolmogorov-Fokker-Planck type. We introduce a  semigroup based notion of nonlocal \emph{carr\'e du champ} which works successfully in situations in which the infinitesimal generator of the semigroup itself does not necessarily possess a gradient. Our results extend and sharpen the original 2004 chain rule due to A. C\'ordoba and D. C\'ordoba.
\end{abstract}

\author{Federico Buseghin}
\address{Dipartimento d'Ingegneria Civile e Ambientale (DICEA)\\ Universit\`a di Padova\\ Via Marzolo, 9 - 35131 Padova,  Italy}

\vskip 0.2in

\email{federico.buseghin@studenti.unipd.it}

\author{Nicola Garofalo}

\address{Dipartimento d'Ingegneria Civile e Ambientale (DICEA)\\ Universit\`a di Padova\\ Via Marzolo, 9 - 35131 Padova,  Italy}

\vskip 0.2in

\email{nicola.garofalo@unipd.it}

\thanks{The second author was supported in part by a Progetto SID (Investimento Strategico di Dipartimento) ``Non-local operators in geometry and in free boundary problems, and their connection with the applied sciences", University of Padova, 2017.}

\maketitle


\section{Introduction and statements of the results}\label{S:intro}

The chain rule for the standard Laplacian states that, given a function $\vf\in C^2(\R)$ and a function $u\in C^2(\Om)$, $\Om\subset \Rn$ an open set, then 
\begin{equation}\label{1}
\Delta \vf(u) = \vf'(u) \Delta u + \vf''(u) |\nabla u|^2.
\end{equation}
As it is well-known, such property, which extends to more general second-order differential operators with non-smooth coefficients, plays a central role in the study of the regularity properties of generalised solutions, see \cite{DG}, \cite{N} and \cite{Mo}, \cite{Mo2}. It is obvious that if $\vf$ is convex, then \eqref{1} implies
\begin{equation}\label{2}
- \Delta \vf(u) \le  \vf'(u) (-\Delta u).
\end{equation}
Consider now the fractional Laplacian defined by 
\begin{equation}\label{laps}
(-\Delta)^s u(x) = \frac{\gamma(n,s)}{2} \int_{\Rn} \frac{2 u(x) - u(x+y) - u(x-y)}{|y|^{n+2s}} dy,
\end{equation}
with normalisation constant given by
\begin{equation}\label{gnsfin}
\gamma(n,s) = \frac{s 2^{2s} \G\left(\frac{n+ 2s}{2}\right)}{\pi^{\frac n2} \G(1-s)}.
\end{equation} 
In \cite[Theorem 1]{CC} the authors proved for this pseudo-differential operator the following  inequality:
\begin{equation}\label{cc}
(-\Delta)^s u^2 \le 2 u (-\Delta)^s u,
\end{equation}
for any $u\in \mathscr S(\Rn)$.
They also presented some important applications of \eqref{cc} to time-decay estimates for viscosity solutions of quasi-geostrophic equations. In the paper \cite{CM} the authors generalised the inequality \eqref{cc} to any convex function $\vf\in C^1(\R)$ and to the fractional powers of the Laplacian on a compact manifold $M$. Precisely, they showed that for any $u\in C^\infty(M)$ one has
 \begin{equation}\label{3}
(-\Delta)^s \vf(u) \le \vf'(u) (-\Delta)^s u,
\end{equation} 
where now $(-\Delta)^s$ is suitably defined.
We quote from \cite{CM}: ``Despite its apparent simplicity its validity is quite surprising given the non-local character of the involved operators".  The inequality \eqref{3} represents a nonlocal version of \eqref{2} and, similarly to its local counterpart, it plays an important role in many problems from the applied sciences involving $(-\Delta)^s$. See for instance the works \cite{CC2}, on the two-dimensional quasi-geostrophic equation, and \cite{CV}, on nonlinear evolution equations with fractional diffusion. 

In this note we generalise these results to the fractional powers of a large class of evolutive hypoelliptic equations and show that, in fact, we can improve on the inequality \eqref{3} and obtain an equality similar to \eqref{1} above. We achieve this by introducing a  semigroup based notion of nonlocal \emph{carr\'e du champ} which works successfully in situations in which the infinitesimal generator of the semigroup itself does not necessarily possess a gradient. A prototypical example of what we have in mind is given by the situation when $\vf(t) = t^2$ treated in \cite{CC}. Given a function $u\in \mathscr S(\Rn)$
consider the Aronszajn-Gagliardo-Slobedetzky $s-$\emph{energy} of $u$ 
\begin{equation}\label{Es}
\mathscr E_{(s)}(u) = \frac{\gamma(n,s)}{4} \int_{\Rn}\int_{\Rn} \frac{(u(x) - u(y))^2}{|x-y|^{n + 2s}} dy dx,
\end{equation}
which defines membership in the fractional Sobolev space $W^{s,2}(\Rn)$, see e.g. \cite{DPV}. The relevance of \eqref{Es} is underscored by the fact that the nonlocal equation $(-\Delta)^s u = 0$ is the Euler-Lagrange equation of the functional $u\to \mathscr E_{(s)}(u)$. Given $u\in \mathscr S(\Rn)$, we have in fact for every $\vf\in \mathscr S(\Rn)$
\begin{equation*}\label{el}
\frac{d}{dt} \mathscr E_{(s)}(u+t \vf)\big|_{t=0}  =  \int_{\Rn}  (-\Delta)^s u(x) \vf(x) dx.
\end{equation*}
This shows that $u$ is a critical point of $\mathscr E_{(s)}$ if and only if $(-\Delta)^s u = 0$. Now, the definition \eqref{Es} of the energy suggests to define for every $x\in \Rn$ the quantity
\begin{equation}\label{gamma2}
\Gamma_{(s)}(u)(x) = \frac{\gamma(n,s)}{2} \int_{\Rn} \frac{(u(x) - u(y))^2}{|x-y|^{n + 2s}} dy.
\end{equation}
One notable property of \eqref{gamma2}  is that it represents a nonlocal version of the P.A. Meyer \emph{carr\'e du champ} which is at the basis of the Bakry-\'Emery gamma calculus. One has in fact the remarkable identity (see \cite[Lemma 20.2]{Gft}) 
\begin{equation}\label{carre}
\Gamma_{(s)}(u) = - \frac 12 [(-\Delta)^s(u^2) - 2 u (-\Delta)^s u].
\end{equation}
Notice that, since from \cite[Lemma 2.3]{B} we know that $\underset{s\to 1^-}{\lim} (-\Delta)^s u(x) = - \Delta u(x)$, a direct consequence of \eqref{carre} is that 
\begin{equation}\label{lim}
\underset{s\to 1^-}{\lim} \Gamma_{(s)}(u) = - \frac 12 \underset{s\to 1^-}{\lim} \left[(-\Delta)^s(u^2)  - 2 u (-\Delta)^s u\right] = - \frac 12 [-\Delta(u^2) + 2 u \Delta u] =  |\nabla u|^2,
\end{equation}
so that $\Gamma_{(s)}(u)$ provides indeed a good nonlocal length of the ``gradient". But the most striking consequence of the identity \eqref{carre} is that it gives the following sharper version of the chain rule \eqref{cc}.

\begin{proposition}\label{P:ft}
Let $0<s<1$. For every $u\in \mathscr S(\Rn)$ one has
\begin{equation}\label{gu}
(-\Delta)^s(u^2)  = 2 u (-\Delta)^s u - 2 \Gamma_{(s)}(u).
\end{equation}
\end{proposition}

In view of \eqref{gamma2}, it is clear that \eqref{gu} trivially implies \eqref{cc}. It is also clear from \eqref{lim} that, if in \eqref{gu} we pass to the limit as $s\nearrow 1^-$, we recover the local identity
\[
\Delta u^2 = 2 u \Delta u + 2 |\nabla u|^2,
\]
which corresponds to the case $\vf(t) = t^2$ of \eqref{1}. 

The first question that we address in the present paper is to what extent Proposition \ref{P:ft} continues to be valid when $\vf(t) = t^2$ is replaced by a generic function $\vf$. Because of the nonlocal nature of $(-\Delta)^s$, we should not expect formula \eqref{gu} to generalise exactly, there is a tail. However, such tail vanishes in the limit as $s\nearrow 1^-$. One has in fact the following result.

\begin{proposition}\label{P:rem}
Let $U\subset \R$ be an interval and suppose that $\vf\in C^{1,1}(U) \cap C^{2,\alpha}_{loc}(U)$, for some $\alpha>0$. For any function $u\in \mathscr S(\Rn)$ such that $u(\Rn)\subset U$, one has
\begin{equation}\label{gugen}
(-\Delta)^s \vf(u(x))  = \vf'(u(x)) (-\Delta)^s u(x) - \vf''(u(x)) \Gamma_{(s)}(u)(x) + \mathscr R_{(s)}(u;\vf)(x),
\end{equation}
where for any $x\in \Rn$ we have 
\begin{equation}\label{limrem}
\underset{s\to 1^-}{\lim} \mathscr R_{(s)}(u;\vf)(x) = 0.
\end{equation}
If $\vf(t) = a t^2 + bt + c$, then we have $\mathscr R_{(s)}(u;\vf) \equiv 0$.
As a consequence of \eqref{gugen}, \eqref{limrem}, we obtain 
\begin{equation}\label{locals}
- \underset{s\to 1^-}{\lim} (-\Delta)^s \vf(u(x))  =  \vf'(u(x)) \Delta u(x) + \vf''(u(x)) |\nabla u(x)|^2.
\end{equation} 
\end{proposition}

Our main objective in this note is generalising Propositions \ref{P:ft} and \ref{P:rem} to the fractional powers of the following class of nonlocal Kolmogorov-Fokker-Planck operators in $\RNu$,  recently studied in \cite{GT, GThls, GTfi, GTiso}:
\begin{equation}\label{K0}
\mathscr K u = \mathscr A u  - \de_t u \overset{def}{=} \operatorname{tr}(Q \nabla^2 u) + <BX,\nabla u> - \de_t u,
\end{equation} 
where the $N\times N$ matrices $Q$ and $B$ have real, constant coefficients, and $Q = Q^\star \ge 0$. 
We assume throughout that $N\ge 2$, and we indicate with $X$ the generic point in $\R^N$, with $(X,t)$ the one in $\RNu$. The class \eqref{K0} was introduced by H\"ormander in his celebrated 1967 hypoellipticity paper \cite{Ho}, where he proved that $\K$ is hypoelliptic if and only if the covariance matrix
\begin{equation}\label{Kt}
K(t) = \frac 1t \int_0^t e^{sB} Q e^{s B^\star} ds
\end{equation}
is invertible (i.e., strictly positive) for every $t>0$. The hypothesis $K(t)>0$ will be henceforth tacitly assumed.
We note that, in the special case when $Q = I_N$ and $B = O_N$, then \eqref{K0} becomes the standard heat operator $H = \Delta - \p_t$ in $\RNu$. One should note that, even in this seemingly simple example, one lacks an obvious notion of ``gradient". In fact, because of the evolutive nature of $H$ a tool like the P.A. Meyer \emph{carr\'e du champ} $\Gamma^H(u) = \frac 12 [H(u^2) - 2 u H u]$ is not effective here since $\G^H(u) = |\nabla u|^2$, which does not provide any control on the time variable $t$.  
The lack of a gradient is not only caused by the time variable. Even for the time-independent operator $\sA$ in \eqref{K0}, we have $\Gamma^\sA(u) = \frac 12 [\sA(u^2) - 2 u \sA u] = <Q\nabla u,\nabla u>$, and this quantity fails to control the directions of the drift, or those of non-ellipticity of $Q$, in the degenerate case.

Since the operators $\mathscr A$ and $\K$ in \eqref{K0} are not variational, with the goal of generalising Propositions \ref{P:ft} and \ref{P:rem} a crucial point is to understand what replaces the nonlocal energy $\Gamma_{(s)}$ in \eqref{gu}, \eqref{gugen}. We introduce a general notion of nonlocal energy which is exclusively semigroup based and therefore prescinds from the existence of a  ``gradient". Specialised to time-independent functions, such energy perfectly recaptures that in  \eqref{gamma2} for the fractional Laplacian, see Proposition \ref{P:tind} below.

In order to introduce the relevant notion we recall the semigroup $P_t = e^{-t \sA}$ defined on a function $f\in \So(\RN)$ by
\begin{equation}\label{Pt}
P_t f(X) = \int_{\RN} p(X,Y,t) f(Y) dY,
\end{equation}
where
\begin{equation}\label{p}
p(X,Y,t) = (4\pi)^{- \frac N2} \left(\operatorname{det}(t K(t))\right)^{-1/2} \exp\left( - \frac{<K(t)^{-1}(Y-e^{tB} X ),Y-e^{tB} X >}{4t}\right)
\end{equation}
is the fundamental solution of the operator $\K$ constructed by H\"ormander in \cite{Ho}. We recall that $\int_{\RN} p(X,Y,t) dY = 1$ for every $X\in \RN$ and $t>0$ (however, one has $\int_{\RN} p(X,Y,t) dX = e^{-t \operatorname{tr}B}$). Using $P_t$ we next consider the evolutive semigroup $P_\tau^\K$ introduced in \cite{GT}. For a function $u\in \S$, we let
\begin{equation}\label{PKtau}
P_\tau^\K u(X,t) = \int_{\RN} p(X,Y,\tau) u(Y,t-\tau) dY.
\end{equation}
For the main properties of the semigroup $\{P_\tau^\K\}_{\tau >0}$ we refer the reader to \cite{GT}. Here, we note that if for $u\in \S$ we define $U((X,t),\tau) = P^\K_\tau u(X,t)$,
then $U\in C^\infty(\RNu\times (0,\infty))$
and it solves the Cauchy problem
$$\begin{cases}
\de_\tau U = \K U \ \ \ \ \ \ \ \ \ \ \ \ \ \ \ \  \  \text{in}\ \RNu\times (0,\infty),
\\
U((X,t),0) = u(X,t)\ \ \ \ \ \ \ (X,t)\in \RNu.
\end{cases}$$
As in \cite{GT}, given $s\in (0,1)$, we now define the nonlocal operator $\Ks$ for $u \in\S$ and $(X,t)\in \RNu$, as follows
\begin{equation}\label{Ks}
\left(-\K\right)^s u(X,t) = - \frac{s}{\G(1-s)} \int_0^\infty \tau^{-1-s} \left[P^\K_\tau u(X,t)- u(X,t)\right] d\tau.
\end{equation}

 We note from \eqref{Pt}, \eqref{PKtau} that, if $u$ does not depend on $t$, i.e., $u(X,t) = v(X)$, then $P_\tau^\K u(X,t)  = P_\tau v(X)$. The next definition is central to this note. It introduces a  semigroup based notion of nonlocal \emph{carr\'e du champ} which works successfully in situations in which the infinitesimal generator of the semigroup itself does not necessarily possess a gradient. 

\begin{definition}\label{nonlocalen}
Given $s\in (0,1)$, we define the nonlocal evolutive \emph{carr\'e du champ} of a function $u:\RNu\to \R$ as 
\[
\G^\K_{(s)}(u)(X,t) = \frac{s}{2\G(1-s)}\int_0^\infty \frac{1}{\tau^{1+s}} P^\K_\tau((u - u(X,t))^2)(X,t) d\tau.
\]
\end{definition}
We notice the obvious consequence $\G^\K_{(s)}(u)(X,t)\ge 0$ of the positivity of  the semigroup $P^\K_\tau$. More importantly, the relevance of Definition \ref{nonlocalen} is connected with the following Besov spaces, see also \cite{GTfi} for the time-independent case.

\begin{definition}\label{D:besov}
For $p\geq 1$ and $\alpha\geq 0$, we define the \emph{evolutive Besov space} $\Bpa$ as the collection of those functions $u\in L^p(\RNu)$, such that the seminorm
\begin{equation*}
\mathscr N^\K_{\alpha,p}(u) = \left(\int_0^\infty  \frac{1}{\tau^{\frac{\alpha p}2 +1}} \int_{\RNu} P_\tau^\K\left(|u - u(X,t)|^p\right)(X,t) dXdt d\tau\right)^{\frac 1p} < \infty.
\end{equation*}
We endow the space $\Bpa$ with the following norm
$$||u||_{\Bpa} \overset{def}{=} ||u||_{L^p(\RNu)} + \mathscr N^\K_{\alpha,p}(u).$$
\end{definition}
We are interested in the situation in which $p = 2$, and $\alpha = s\in(0,1)$. We remark explicitly that, in such case, we have
\[
\mathscr N^\K_{s,2}(u)^{2} = \frac{2\G(1-s)}{s}\int_{\RNu} \G_{(s)}^\K(u)(X,t) dX dt.
\]
We have the following result that shows that, specialised to functions which do not depend on $t$, and to the standard heat operator,  we recover from $\G^\K_{(s)}(u)(X,t)$ the \emph{carr\'e du champ} \eqref{carre}, \eqref{gamma2}.

\begin{proposition}\label{P:tind}
Suppose that $Q = I_N$ and $B = O_N$, and thus $\K = \Delta - \p_t$, and let $u(X,t) = v(X)$. Then, 
\[
\G^\K_{(s)}(u)(X,t) = \G_{(s)}(v)(X) = \frac{\gamma(N,s)}{2} \int_{\RN} \frac{(v(X) - v(Y))^2}{|X-Y|^{N + 2s}} dY,
\]
see \eqref{gamma2}.
\end{proposition}

With Definition \ref{nonlocalen} in hands, we are now ready to state the evolutive counterpart of Proposition \ref{P:ft}.

\begin{proposition}\label{P:evcr}
Let $u\in \S$. Then, for every $(X,t)\in \RNu$ we have
\begin{equation}\label{Ksuq}
(-\mathscr{K})^{s}(u^{2})(X,t) = 2u(X,t)(-\mathcal{K})^{s} u(X,t)-2\Gamma^\K_{(s)}(u)(X,t).
\end{equation}
In particular, the following extension of \eqref{cc} holds:
\begin{equation}\label{Ksuqle}
(-\mathscr{K})^{s}(u^{2})(X,t) \le 2u(X,t)(-\mathscr{K})^{s} u(X,t).
\end{equation}
More in general, for every convex function $\vf\in C^1(\R)$, one has
\begin{equation}\label{Ksvfle}
(-\mathscr{K})^{s} \vf(u)(X,t) \le \vf'(u)(X,t)(-\mathscr{K})^{s} u(X,t).
\end{equation}
\end{proposition}

We next consider the appropriate generalisation of  Proposition \ref{P:rem}.

\begin{proposition}\label{P:remK}
Let $U\subset \R$ be an interval and suppose that $\vf\in C^{1,1}(U) \cap C^{2,\alpha}_{loc}(U)$, for some $\alpha>0$. For any function $u\in \mathscr{S}(\mathbb{R}^{N+1})$ such that $u(\RNu)\subset U$, one has
\begin{align}\label{gugenK}
(-\K)^s \vf(u(X,t)) & = \vf'(u(X,t)) (-\K)^s u(X,t) - \vf''(u(X,t)) \G^\K_{(s)}(u)(X,t) 
\\
& + \mathscr R^\K_{(s)}(u;\vf)(X,t),
\notag
\end{align}
where for any $(X,t)\in \RNu$ we have 
\begin{equation}\label{limremK}
\underset{s\to 1^-}{\lim} \mathscr R^\K_{(s)}(u;\vf)(X,t) = 0.
\end{equation}
When $\vf(t) = a t^2 + bt + c$, we have $\mathscr R^\K_{(s)}(u;\vf) \equiv 0$.
\end{proposition}

In Section \ref{S:proofs} we present the proofs of Propositions \ref{P:tind}, \ref{P:evcr} and \ref{P:remK}. In view of Proposition \ref{P:tind}, Proposition \ref{P:rem} is contained in the more general Proposition \ref{P:remK}, and we thus omit its proof.

\vskip 0.2in

\noindent \textbf{Acknowledgment:} We thank Giulio Tralli for his interest in this note and for helpful discussions.


\section{Proof of the results}\label{S:proofs}

In this section we present the proofs of the results.  
We begin with the
 
\begin{proof}[Proof of Proposition \ref{P:tind}]
Under the assumptions of the proposition we know that $p(X,Y,t)=(4\pi t)^{-\frac{N}{2}}\exp\big(-\frac{|X-Y|^{2}}{4t}\big)$.
Since from \eqref{Pt}, \eqref{PKtau} we have $P^\K_\tau((u - u(X,t))^2)(X,t) = P_\tau((v - v(X))^2)(X)$, we find 
\begin{align*}
\G^\K_{(s)}(u)(X,t) & = \frac{s}{2\Gamma(1-s)}\int_{0}^{\infty} \tau^{-1-s}P_{\tau}((v-v(X))^{2})(X)d\tau
\\
&= \frac{s 2^{-N} \pi^{-\frac{N}{2}}}{2\Gamma(1-s)}\int_{\mathbb{R}^{N}} (v(Y)-v(X))^{2} \int_{0}^{+\infty}\tau^{-1-s-\frac{N}{2}}\exp\big(-\frac{|X-Y|^{2}}{4\tau}\big) d\tau dY.
\end{align*}
Now, a simple computation gives
\begin{align*}
&\int_{0}^{+\infty}\tau^{-1-s-\frac{N}{2}}\exp\big(-\frac{|X-Y|^{2}}{4\tau}\big)d\tau = \frac{2^{N+2s}}{|X-Y|^{N+2s}}\Gamma\left(\frac{N+2s}2\right).
\end{align*}
Substituting this identity in the above equation,  recalling \eqref{gnsfin} which gives
\begin{align*}
\gamma(N,s)=\frac{s 2^{2s} \Gamma(s+\frac{N}{2})}{\pi^{\frac{N}{2}}\Gamma(1-s)},
\end{align*}
and keeping \eqref{gamma2} in mind,
we reach the desired conclusion $\G^\K_{(s)}(u)(X,t) = \G_{(s)}(v)(X)$.

\end{proof}

Next, we present the 

\begin{proof}[Proof of Proposition \ref{P:evcr}]
We only prove \eqref{Ksvfle}, since \eqref{Ksuq} will follow from Proposition \ref{P:remK}, and \eqref{Ksuqle} is a trivial consequence of \eqref{Ksvfle}.
By the assumption of convexity of $\vf$ we have for any $\gamma, \sigma \in \mathbb{R}$,
\[
\phi'(\sigma)(\gamma -\sigma) \le \phi(\gamma)-\phi(\sigma).
\]	
Applying this inequality with $\gamma=u(Y,t-\tau)$ and $\sigma=u(X,t)$, we obtain
\begin{align*}
	P_{\tau}^{\K}\phi\circ u(X,t) - \phi\circ u(X,t)& = \int_{\mathbb{R}^{N}}p(X,Y,\tau)\big(\phi(u(Y,t-\tau))-\phi(u(X,t))\big)dY\geq\\
	&\geq \phi'(u(X,t))\int_{\mathbb{R}^{N}}p(X,Y,\tau)\big(u(Y,t-\tau)-u(X,t)\big)dY=\\
	&=\phi'(u(X,t))\big(P_{\tau}^{\K}u(X,t)-u(X,t)\big).
\end{align*}	
Multiplying both sides of the latter inequality by $- \frac{s}{\G(1-s)}\tau^{-1-s}$ and integrating in $\tau$ over $(0,\infty)$, if we keep \eqref{Ks} in mind we immediately obtain the desired conclusion \eqref{Ksvfle}.

\end{proof}

Finally, we give the 

\begin{proof}[Proof of Proposition \ref{P:remK}] For every $(X,t)\in \RNu$ fixed, we have
\begin{align}\label{one}
P_{\tau}^{\K}\phi(u)(X,t) -\phi(u)(X,t)&=\int_{\mathbb{R}^{N}}p(X,Y,\tau)\big(\phi(u(Y,t-\tau))-\phi(u(X,t))\big)dY.
\end{align}	
By the Taylor expansion of $\phi(u)$ we can write
\begin{align*}
& \phi(u(Y,t-\tau))-\phi(u(X,t))=\phi'(u(X,t))(u(Y,t-\tau)-u(X,t))
\\ 
& +\frac 12 \phi''(u(X,t))(u(Y,t-\tau)-u(X,t))^{2}
+ \frac 12 (\phi''(\tilde{u}) - \phi''(u(X,t)))(u(Y,t-\tau)-u(X,t))^{2},
\end{align*}	
where $\tilde u$ is a point (depending on $X, Y, t, \tau$) between $u(X,t)$ and $u(Y,t - \tau)$. Substituting in the above identity, we find 
\begin{align*}
 P_{\tau}^{\K}\phi(u)(X,t)-\phi(u)(X,t) & =\phi'(u)(X,t)(P^\K_\tau u(X,t) - u(X,t))  
\\
& + \frac 12 \phi''(u)(X,t) P^\K_\tau (u -u(X,t))^{2}(X,t)
\\
& + \frac{1}{2}\int_{\mathbb{R}^{N}}p(X,Y,\tau)(\phi''(\tilde{u})-\phi''(u(X,t)))(u(Y,t-\tau)-u(X,t))^{2}dY.
\end{align*}
Multiplying both sides of the latter equation by $- \frac{s}{\G(1-s)}\tau^{-1-s}$ and integrating in $\tau$ over $(0,\infty)$, if we keep  \eqref{Ks} and \eqref{nonlocalen} in mind we obtain \eqref{gugenK}, where we have let
\[
\mathscr R^\K_{(s)}(u;\vf)(X,t) = -\frac{s}{2\Gamma(1-s)}\int_{0}^{+\infty}\frac{1}{\tau^{1+s}}\int_{\mathbb{R}^{N}}p(X,Y,\tau)(\phi''(\tilde{u})-\phi''(u(X,t)))(u(Y,t-\tau)-u(X,t))^{2}dY.
\]
From the latter equation it is obvious that, when $\vf(t) = a t^2 + bt + c$, then $\vf'' \equiv 2a$, and therefore we have $\mathscr R^\K_{(s)}(u;\vf)(X,t) \equiv 0$. Combined with \eqref{gugenK}, this proves in particular the chain rule \eqref{Ksuq}.
Finally, we need to show \eqref{limremK}.
With this objective in mind, we write
\begin{align*}
\mathscr R^\K_{(s)}(u;\vf)(X,t)&=-\frac{s}{2\Gamma(1-s)}\int_{0}^{1} \frac{1}{\tau^{1+s}}\int_{\mathbb{R}^{N}}p(X,Y,\tau)(\phi''(\tilde{u})-\phi''(u(X,t)))(u(Y,t-\tau)-u(X,t))^{2}dYd\tau
\\
&-\frac{s}{2\Gamma(1-s)}\int_{1}^{\infty} \frac{1}{\tau^{1+s}}\int_{\mathbb{R}^{N}}p(X,Y,\tau)(\phi''(\tilde{u})-\phi''(u(X,t)))(u(Y,t-\tau)-u(X,t))^{2}dYd\tau
\\
&=I(X,t;s)+II(X,t;s).
\end{align*}
Since by assumption $\vf''\in L^\infty(U)$, using the fact that $\int_{\RN}p(X,Y,\tau) dY = 1$, we easily find for every $(X,t)\in \RNu$,
\begin{align*}
|II(X,t;s)|\le&\frac{4s}{\Gamma(1-s)} ||\phi''||_{\infty}||u||_{\infty} ^{2}\int_{1}^{\infty} \frac{1}{\tau^{1+s}}\int_{\mathbb{R}^{N}}p(X,Y,\tau)dY d\tau
\\
=&\frac{4s}{\Gamma(1-s)} ||\phi''||_{\infty}||u||_{\infty} ^{2}\int_{1}^{\infty} \frac{1}{\tau^{1+s}}d\tau =  \frac{4}{\Gamma(1-s)} ||\phi''||_{\infty}||u||_{\infty} ^{2}\ \longrightarrow\ 0
\end{align*}
when $s\rightarrow1^{-}$. Here, we have used the fact that $\G(1-s)\cong 1/(1-s)$ as $s\to 1^-$. Estimating $I(X,t;s)$ requires some additional care. Because $u\in \mathscr{S}(\mathbb{R}^{N+1})$, Taylor's formula gives
\[
|u(Y,t-\tau) - u(X,t)| \le  |\nabla_{(X,t)} u(\tilde{Y},\tilde{\tau})| (|Y-X| + \tau) \le C_{1} (|Y-X| + \tau),
\]
where $C_{1}=||\nabla_{(X,t)}u||_{\infty}$ and $\tilde Y$, $\tilde \tau$ respectively are a point on the segment joining $X$ to $Y$, and a number in the interval joining $t$ and $t-\tau$. The hypothesis $\phi \in C^{2,\alpha}_{loc}(U)$ thus gives
\begin{align*}
	|\phi''(\tilde{u})-\phi''(u(X,t))|&\le C_{2} |u(Y,t-\tau)-u(X,t)|^{\alpha} \le C_{3} (|Y-X|^{\alpha}+\tau^{\alpha}),
\end{align*}
where $C_{2}, C_{3}$ are two positive constants. Then, there exist a consant $C_4>0$ such that
\begin{align*}
   |u(Y,t-\tau) - u(X,t)|^{2}\,|\phi''(\tilde{u})-\phi''(u(X,t))|\le C_4 \big(|Y-X|^{2+\alpha}+|Y-X|^{2}\tau^{\alpha}+|Y-X|^\alpha \tau^{2}+\tau^{2+\alpha}\big).
\end{align*}
Using this, and \eqref{p}, we find 
\begin{align*}
& |I(X,t;s)|  \le \frac{s C_4}{\G(1-s)}\int_{0}^{1}\frac{1}{\tau^{1+s}\left(\operatorname{det}(\tau K(\tau))\right)^{1/2}}\int_{\mathbb{R}^{N}} \exp\left( - \frac{|K(\tau)^{-1/2}(Y-e^{tB} X)|^2}{4\tau}\right)
\\
& \big(|Y-X|^{2+\alpha}+|Y-X|^{2}\tau^{\alpha}+|Y-X|^{\alpha}\tau^{2}+\tau^{2+\alpha}\big)dYd\tau.
\end{align*}
We now make the change of variable $Z=K^{-1/2}(\tau)(X-e^{-\tau B}Y)/\sqrt \tau$ in the integral over $\RN$. This gives $Y = e^{\tau B} (X - (\tau K(\tau))^{1/2}Z)$, and therefore $dY = (\det(\tau K(\tau)))^{1/2} e^{\tau \tr B}dZ$.
Notice also that, in terms of the new variable $Z$, one has
\begin{align*}
|Y-X| &  = |(e^{\tau B} - I_N)X + e^{\tau B}(\tau K(\tau))^{1/2}Z| \le |(e^{\tau B} - I_N)X|+|e^{\tau B}(\tau K(\tau))^{1/2}Z|.
\end{align*}
Since for $0<\tau<1$ we easily find
\[
|(e^{\tau B} - I_N)X| \le C |X| \tau \le C |X| \tau^{1/2},\ \ \ \ \ |e^{\tau B}(\tau K(\tau))^{1/2}Z|\le C |Z| \tau^{1/2},
\]
for some constant $C>0$, we have
\[
|(e^{\tau B} - I_N)X + e^{\tau B}(\tau K(\tau))^{1/2}Z| \le C \max\{1,|X|\}(1+|Z|) \tau^{1/2}.
\]
It is then clear that there exists a constant $C_5 >0$, depending on the point $X$, and on $N,\alpha$, such that, in terms of the new variable $Z$, one has for all $Z\in \RN$ and $\tau\in (0,1)$ 
\[
|Y-X|^{2+\alpha}+|Y-X|^{2}\tau^{\alpha}+|Y-X|^{\alpha}\tau^{2}+\tau^{2+\alpha} \le C_5 (1+ |Z|^{2+\alpha}) \tau^{1+\frac{\alpha}2}.
\]
We conclude from this estimate that
\begin{align*}
|I(X,t;s)| & \le \frac{s C^{\star\star}}{\G(1-s)}\int_{0}^{1}\frac{\tau^{1+\frac{\alpha}2}}{\tau^{1+s}} d\tau \int_{\mathbb{R}^{N}} (1+ |Z|^{2+\alpha}) e^{- |Z|^2} dZ,
\\
& \le \frac{s C_6}{\G(1-s)}\frac{1}{(1-s) + \frac{\alpha}2}\ \longrightarrow\ 0,
\end{align*}
as $s\to 1^-$. This proves \eqref{limremK}.

\end{proof}

\bibliographystyle{amsplain}

\end{document}